\documentclass[11pt,reqno]{amsart}

\usepackage[utf8]{inputenc}
\usepackage[T1]{fontenc}
\usepackage{lmodern}
\usepackage{microtype} 

\usepackage{amsmath,amssymb,amsfonts,amsthm}
\usepackage{mathtools}
\usepackage{mathrsfs}

\usepackage{graphicx}
\usepackage{booktabs} 
\usepackage{array}

\usepackage[margin=1in]{geometry}

\theoremstyle{plain}
\newtheorem{theorem}{Theorem}[section]

\theoremstyle{definition}
\newtheorem{definition}[theorem]{Definition}
\newtheorem{example}[theorem]{Example}

\numberwithin{equation}{section}

\usepackage{xcolor}
\definecolor{winered}{rgb}{0.5,0,0}
\definecolor{darkblue}{rgb}{0,0,0.6}

\usepackage[
    pdftex,
    colorlinks=true,
    linkcolor=darkblue,
    citecolor=winered,
    urlcolor=darkblue,
    hypertexnames=false 
]{hyperref}

\title[Discrete Prüfer Approach to Eigenvalue Estimates]{A Discrete Prüfer Transformation Approach to Sturm--Liouville Difference Equations and Eigenvalue Estimation}

\author{F. Ay\c{c}a \c{C}etinkaya}
\address{Department of Mathematics, University of Tennessee at Chattanooga, 37403 Chattanooga, TN, USA.}
\email{fatmaayca-cetinkaya@utc.edu}

\author{Kürşat Er}
\address{Mersin University, Department of Mathematics, 33343 Mersin, T\"{u}rkiye.}
\email{erkursat89@gmail.com}
\author{Hamza Menken}
\address{Mersin University, Department of Mathematics, 33343 Mersin, T\"{u}rkiye.}
\email{hmenken@mersin.edu.tr}
\date{\today}

\begin{document}

\begin{abstract}
In this paper, we study regular second-order Sturm--Liouville difference equations using the discrete Prüfer transformation. By representing solutions in amplitude and phase coordinates, we analyze an exact algebraic phase system that guarantees unique, monotonic phase tracking and preserves classical oscillation properties. Using this theoretical foundation, we develop a Prüfer-based numerical shooting method to compute eigenvalues for discrete boundary value problems. To initialize the root-finding algorithm, we apply Gershgorin's theorem to the difference operator to establish mathematically guaranteed starting search intervals. Numerical experiments on classical benchmark problems demonstrate that the proposed method effectively isolates the discrete spectrum and converges to the exact continuous eigenvalues with second-order $\mathcal{O}(h^2)$ accuracy.
\end{abstract}

\maketitle

\section{Introduction}
One important class of differential equations is the Sturm--Liouville equation, which arises naturally in mathematical physics, for example through the method of separation of variables applied to the heat or wave equations. These problems play a central role in understanding eigenvalue distributions, vibration modes, and heat flow. In recent years, interest has grown in studying discrete analogs of Sturm–Liouville problems \cite{agarwal2000difference, bohner1996disconjugacy, bohner1997asymptotic, bohner2017bessel, elaydiintroduction, kelley2001difference}. Formulating these problems in a discrete setting not only broadens their applicability but also enables the development of efficient computational methods for solving associated eigenvalue problems.

The Pr\"ufer substitution is a classical technique used in the analysis of second-order linear differential equations, particularly Sturm--Liouville problems:
\[
(P(x) y'(x))' + Q(x) y(x) = 0, \quad x \in [a, b],
\]
It involves expressing the solution in polar coordinates, transforming the function and its derivative into an amplitude $r(x)$ and a phase $\theta(x)$ via the substitution
\[
y(x) = r(x) \sin \theta(x),
\]
\[
P(x) y'(x) = r(x) \cos \theta(x).
\]
This change of variables reduces the original second-order equation to a coupled first-order system
\[
\frac{d\theta}{dx} = \frac{\cos^2 \theta}{P(x)} + Q(x) \sin^2 \theta,
\]
\[
\frac{dr}{dx} = r \sin \theta \cos \theta \left( \frac{1}{P(x)} - Q(x) \right),
\]
where the phase equation governs the evolution of the argument $\theta(x)$, and the amplitude equation describes the growth or decay of the solution. One key advantage of the Pr\"ufer transformation is that the zeros of the solution correspond to values where $\theta(x)=n \pi$, making it a powerful tool for oscillation theory, eigenvalue estimation, and nodal analysis. It has been extensively used to study the qualitative behavior of solutions, count eigenvalues, and establish comparison theorems in continuous Sturm--Liouville theory \cite{zbMATH03450627, zbMATH03190499, zbMATH03109695, zbMATH06054089, zbMATH02219859}.

In this paper, we study a second-order Sturm--Liouville difference equation of the form
\begin{equation} \label{dif_eq} 
\Delta \big(r_k \Delta x_k\big) + p_k x_{k+1} = 0,
\end{equation}
where $r_k$ and $p_k$ are real-valued sequences defined for all integers $k$, with $r_k \neq 0$ for every $k$. The operator $\Delta$ denotes the forward difference operator, defined by 
\begin{equation} \label{for_dif}
\Delta x_k=x_{k+1}-x_{k}.
\end{equation}

Following \cite{zbMATH01602173}, we introduce a change of variables that represents the solution in terms of its amplitude $\varrho_k$ and phase $\varphi_k$: 
\begin{align} \label{PT_x}
x_k &= \varrho_k\sin \varphi_k,\\ \label{PT_rx}
r_k \Delta x_k &= \varrho_k \cos \varphi_k.
\end{align}
This substitution, known as the \emph{discrete Prüfer transformation}, will be used to analyze the qualitative behavior of the solutions.

\section{The Discrete Prüfer Transformation}
The following theorem provides a discrete analogue of the classical Prüfer differential equations. It establishes the relations satisfied by the amplitude and phase variables defined through the discrete Prüfer transformation.
\begin{theorem}\cite[Theorem 2.1]{zbMATH01602173}
If $x$ is a nontrivial solution of \eqref{dif_eq} and of $\varrho$ and $\varphi$ are defined by \eqref{PT_x} and \eqref{PT_rx}, then the equations
\begin{align} \label{rho_de}
\Delta \varrho_k=&\varrho_k \Bigg[\dfrac{1}{r_k}\cos \varphi_k \sin \varphi_{k+1}-p_k \sin \varphi_k \cos \varphi_{k+1}-\dfrac{p_k}{r_k}\cos \varphi_k \cos \varphi_{k+1}-\dfrac{\big(\Delta \sin \varphi_k\big)^2 +\big(\Delta \cos \varphi_k\big)^2}{2}\Bigg], \\ 
\sin \Delta \varphi_k=&\dfrac{1}{r_k} \cos \varphi_k \cos \varphi_{k+1}+p_k \sin \varphi_k \sin\varphi_{k+1}+\dfrac{p_k}{r_k} \cos \varphi_k \sin \varphi_{k+1} \label{phi_de}
\end{align}
hold true. 
\end{theorem}
The system consisting of equations \eqref{rho_de} and \eqref{phi_de} is called the \textit{Prüfer system} associated with the difference equation \eqref{dif_eq}. This system provides a complete characterization of the solutions to the original equation, as formalized in the following theorem.
\begin{theorem}
Suppose that \eqref{dif_eq}, \eqref{rho_de}, and \eqref{phi_de} are valid. If $\varrho_k$ and $\varphi_k$ are solutions to \eqref{rho_de} and \eqref{phi_de}, respectively, then the function $x_k$ defined by \eqref{PT_x} solves \eqref{dif_eq}. Conversely, if $x_k$ is a nontrivial solution to \eqref{dif_eq}, then there exist sequences $\varrho_k$ and $\varphi_k$ satisfying \eqref{rho_de} and \eqref{phi_de} such that \eqref{PT_x}, \eqref{PT_rx} hold with $\varrho_k \neq 0.$
\end{theorem}
\begin{proof}
Assume $\varrho_k$ and $\varphi_k$ satisfy \eqref{rho_de} and \eqref{phi_de}, respectively and $x_k$ is defined by \eqref{PT_x}. We want to show \eqref{dif_eq} is satisfied. From \eqref{PT_x} and \eqref{PT_rx}, the left side of \eqref{dif_eq} becomes
\[
\big(r_{k+1}\Delta x_{k+1}-r_k \Delta x_k\big)+p_k x_{k+1}=\big(\varrho_{k+1}\cos\varphi_{k+1}-\varrho_k \cos\varphi_k\big)+p_k \varrho_{k+1}\sin \varphi_{k+1}.
\]
We must prove
\[
\varrho_{k+1} \big(\cos \varphi_{k+1}+p_k \sin \varphi_{k+1}\big)=\varrho_k \cos \varphi_k.
\]
Using the identity
\[
\dfrac{\big(\Delta \sin \varphi_k\big)^2+\big(\Delta \cos \varphi_k\big)^2}{2}=1-\cos \Delta \varphi_k
\]
to simplify \eqref{rho_de} we have
\begin{equation} \label{rho}
\varrho_{k+1}=\varrho_k \bigg[\cos \Delta \varphi_k +\dfrac{1}{r_k}\cos \varphi_k \sin \varphi_{k+1}-p_k \sin \varphi_k \cos \varphi_{k+1}-\dfrac{p_k}{r_k}\cos \varphi_k \cos \varphi_{k+1}\bigg].
\end{equation}
Multiplying \eqref{rho} by $\cos \varphi_{k+1}+p_k \sin \varphi_{k+1}$ gives 
\begin{align} \nonumber
\varrho_{k+1} \big(\cos \varphi_{k+1} + p_k \sin \varphi_{k+1}\big) &= \varrho_k \bigg[\cos \varphi_{k+1} \cos \Delta \varphi_k +p_k \sin \varphi_{k+1} \cos \Delta \varphi_k \\ \nonumber
&+ \dfrac{1}{r_k} \cos \varphi_k \cos \varphi_{k+1} \sin \varphi_{k+1} +\dfrac{p_k}{r_k} \cos \varphi_k  \sin^2 \varphi_{k+1} \\ \nonumber
&-\dfrac{p_k}{r_k}\cos \varphi_k \cos^2 \varphi_{k+1}-\dfrac{p_k^2}{r_k}\cos \varphi_k \cos \varphi_{k+1} \sin \varphi_{k+1}\\
&-p_k \sin \varphi_k \cos^2 \varphi_{k+1}-p_k^2 \sin \varphi_k \cos \varphi_{k+1} \sin \varphi_{k+1}\bigg]. \label{rho_multiplied}
\end{align}
We will show the bracket equals $\cos \varphi_k$. Multiplying \eqref{phi_de} by $\sin \varphi_{k+1}$ gives
\begin{equation} \label{phi_de_multiplied}
\sin \varphi_{k+1} \sin \Delta \varphi_k-p_k \sin \varphi_k \sin^2 \varphi_{k+1}= \dfrac{1}{r_k} \cos \varphi_{k} \cos \varphi_{k+1} \sin \varphi_{k+1}+\dfrac{p_k}{r_k} \cos \varphi_k \sin^2 \varphi_{k+1}. 
\end{equation}
Substituting the left hand side of \eqref{phi_de_multiplied} into \eqref{rho_multiplied} and making necessary calculations we arrive 
\begin{align} \label{calculations}
\varrho_{k+1} \big(\cos \varphi_{k+1} + p_k \sin \varphi_{k+1}\big) =&\varrho_k\bigg[ \cos \varphi_k+ p_k \cos \varphi_{k+1} \sin \Delta \varphi_k\\ \nonumber
&-\dfrac{p_k}{r_k} \cos \varphi_k \cos \varphi_{k+1} \big(\cos \varphi_{k+1}+p_k \sin \varphi_{k+1}\big)-p_k^2 \sin \varphi_k \cos \varphi_{k+1} \sin \varphi_{k+1}\bigg].
\end{align}
Now multiply \eqref{phi_de} by $p_k \cos \varphi_{k+1} $
\begin{equation} \label{almost_there}
p_k \cos \varphi_{k+1} \sin \Delta \varphi_k=\dfrac{p_k}{r_k} \cos \varphi_k \cos \varphi_{k+1} \big(\cos \varphi_{k+1}+p_k \sin \varphi_{k+1}\big)+p_k^2 \sin \varphi_k \cos \varphi_{k+1} \sin \varphi_{k+1}
\end{equation}
Substituting \eqref{almost_there} into \eqref{calculations} gives us the step needed to conclude 
\[
\varrho_{k+1} \big(\cos \varphi_{k+1}+p_k \sin \varphi_{k+1}\big)=\varrho_k \cos \varphi_k
\]
and hence \eqref{dif_eq} is satisfied:
\[
\Delta \big(r_k \Delta x_k\big)+p_k x_{k+1}=0.
\]
Now, we prove the other side. We assume $x_k$ is a nontrivial solution of \eqref{dif_eq}. We will construct $\varrho_k >0$ and $\varphi_k$ so that \eqref{PT_x} and \eqref{PT_rx} hold and then derive the Prüfer system \eqref{rho_de}, \eqref{phi_de}. \newline
\indent For each $k$, set
\[
\varrho_k := \sqrt{x_k^2+\big(r_k \Delta x_k\big)^2} \geq 0.
\]
If $\varrho_k >0$ we define $\varphi_k$ by 
\[
\sin \varphi_k:=\dfrac{x_k}{\varrho_k}, \quad
\cos \varphi_k := \dfrac{r_k \Delta x_k}{\varrho_k}.
\]
For any nontrivial solution $x_k$, we have $\varrho_k >0$ for all $k$. Indeed, suppose $\varrho_{k_0}=0$ for some $k_0$. Then $x_{k_0}=0$ and $r_{k_0}\Delta x_{k_0}=0$. Since $r_{k_0} \neq 0$, this gives $\Delta x_{k_0}=0$, hence $x_{{k_0}+1}=x_{k_0}=0.$ Plugging $x_{{k_0}+1}=0$ into \eqref{dif_eq} at $k=k_0$ we get $r_{k_1}\Delta x_{k_1}=r_{k_0}\Delta x_{k_0}=0$ with $k_1=k_0+1$, and so $\Delta x_{k_1}=0,$ hence $x_{{k_1}+1}=0.$ Inducting forward shows $x_k =0$ for all $k \geq k_0$. Running the same argument backwards (apply \eqref{dif_eq} at $k-1$) shows $x_k =0$ for all $k \leq k_0$. Thus, $x$ is identically zero, contradicting nontriviality. So, $\varrho_k$ and $\varphi_k$ are well defined for all $k$, and by construction \eqref{PT_x} and \eqref{PT_rx} hold. \newline
\indent From $r_k \Delta x_k =\varrho_k \cos \varphi_k$ we get $\Delta x_k =\dfrac{\varrho_k}{r_k}\cos \varphi_k.$ Thus,
\begin{equation} \label{x_k1}
x_{k+1}=x_k +\Delta x_k=\varrho_k \sin \varphi_k+\dfrac{\varrho_k}{r_k}\cos \varphi_k=\varrho_k \bigg(\sin \varphi_k +\dfrac{1}{r_k}\cos \varphi_k\bigg).
\end{equation}
Using \eqref{dif_eq},
\begin{align} \label{deltax_k1}
r_{k+1} \Delta x_{k+1}&=r_k \Delta x_k -p_k x_{k+1}=\varrho_k \cos \varphi_k -p_k \varrho_k \bigg(\sin \varphi_k +\dfrac{1}{r_k} \cos \varphi_k\bigg)\\ \nonumber
&=\varrho_k \Bigg(\bigg(1-\dfrac{p_k}{r_k}\bigg)\cos \varphi_k-p_k \sin \varphi_k\Bigg). 
\end{align}
By definition, 
\[
x_{k+1}=\varrho_{k+1}\sin \varphi_{k+1}, \quad r_{k+1} \Delta x_{k+1}=\varrho_{k+1} \cos \varphi_{k+1}.
\]
Multiplying the first one by $\sin \varphi_{k+1}$ and the second one by $\cos \varphi_{k+1}$ and adding the resulting equations gives
\begin{equation} \label{rho_k1}
x_{k+1} \sin \varphi_{k+1}+\big(r_{k+1} \Delta x_{k+1}\big)\cos \varphi_{k+1}=\varrho_{k+1}.
\end{equation}
Substituting \eqref{x_k1} and \eqref{deltax_k1} into \eqref{rho_k1} gives
\begin{align*}
\varrho_{k+1}&=\Bigg[\varrho_k \bigg(\sin \varphi_k +\dfrac{1}{r_k} \cos \varphi_k\bigg)\Bigg]\sin \varphi_{k+1}+\Bigg[\varrho_k \Bigg(\bigg(1-\dfrac{p_k}{r_k}\bigg)\cos \varphi_k-p_k \sin \varphi_k\Bigg)\Bigg]\cos \varphi_{k+1}\\
&=\varrho_k \bigg[\sin \varphi_k \sin \varphi_{k+1}+\cos \varphi_k \cos \varphi_{k+1}+\dfrac{1}{r_k} \cos \varphi_k \sin\varphi_{k+1} -\dfrac{p_k}{r_k} \cos \varphi_k \cos \varphi_{k+1} -p_k \sin \varphi_k \cos \varphi_{k+1}\bigg]\\
&= \varrho_k \bigg[ \cos \Delta \varphi_k +\dfrac{1}{r_k} \cos \varphi_k \sin\varphi_{k+1} -p_k \sin \varphi_k \cos \varphi_{k+1}-\dfrac{p_k}{r_k} \cos \varphi_k \cos \varphi_{k+1}\bigg]\\
&= \varrho_k \bigg[1-\dfrac{\big(\Delta \sin \varphi_k\big)^2+\big(\Delta \cos \varphi_k\big)^2}{2}+\dfrac{1}{r_k} \cos \varphi_k \sin \varphi_{k+1}-p_k \sin \varphi_k \cos \varphi_{k+1}-\dfrac{p_k}{r_k}\cos \varphi_k \cos \varphi_{k+1}\bigg]
\end{align*}
thus
\[
\Delta \varrho_k= \varrho_k \bigg[ \dfrac{1}{r_k} \cos \varphi_k \sin \varphi_{k+1} -p_k \sin \varphi_k \cos \varphi_{k+1} -\dfrac{p_k}{r_k} \cos \varphi_k \cos \varphi_{k+1} -\dfrac{\big(\Delta \sin \varphi_k\big)^2+\big(\Delta \cos \varphi_k\big)^2}{2}\bigg]
\] which is exactly \eqref{rho_de}.\newline
\indent Now, to obtain \eqref{phi_de} start from the trigonometric identity
\[
\sin \Delta \varphi_k=\sin \big(\varphi_{k+1}- \varphi_k\big)=\sin \varphi_{k+1}\cos \varphi_k-\cos \varphi_{k+1}\sin \varphi_k.
\]
Replacing each sine and cosine function by the definitions in terms of $x$, $\Delta x$, and $\varrho$:
\[
\sin \varphi_{k+1}=\dfrac{x_{k+1}}{\varrho_{k+1}}, \ \cos\varphi_{k+1}=\dfrac{r_{k+1}\Delta x_{k+1}}{\varrho_{k+1}}, \ \cos\varphi_{k}=\dfrac{r_{k}\Delta x_{k}}{\varrho_{k}}, \ \sin \varphi_{k}=\dfrac{x_{k}}{\varrho_{k}}
\]
gives 
\begin{equation} \label{deltaphik}
\sin \Delta \varphi_k = \dfrac{r_k x_{k+1}\Delta x_k -r_{k+1}x_k \Delta x_{k+1}}{\varrho_k \varrho_{k+1}}.
\end{equation}
The numerator in \eqref{deltaphik} simplifies into $r_k (\Delta x_k)^2+p_k x_k x_{k+1}.$ Plugging this into \eqref{deltaphik} we have 
\begin{equation} \label{sindeltaphik}
\varrho_k \varrho_{k+1}\sin \Delta \varphi_k =r_k (\Delta x_k)^2+p_k x_k x_{k+1}.
\end{equation}
Now, we will write the right-side of \eqref{sindeltaphik} in terms of $\varrho$ and $\varphi$. Since $r_k \Delta x_k=\varrho_k \cos \varphi_k$, $\Delta x_k =\dfrac{\varrho_k}{r_k} \cos \varphi_k$, and $r_k (\Delta x_k)^2 =\dfrac{\varrho^2_k}{r_k}\cos^2 \varphi_k.$ Also, $x_k=\varrho_k \sin \varphi_k$ and $x_{k+1}=\varrho_{k+1} \sin \varphi_{k+1}$, so $p_k x_k x_{k+1}=p_k \varrho_k \varrho_{k+1} \sin \varphi_k \sin \varphi_{k+1}.$ Therefore \eqref{sindeltaphik} becomes 
\[
\varrho_k \varrho_{k+1} \sin \Delta \varphi_k =\dfrac{\varrho_k^2}{r_k} \cos^2 \varphi_k+p_k \varrho_k \varrho_{k+1} \sin \varphi_k \sin \varphi_{k+1}.
\]
Dividing both sides by $\varrho_k \varrho_{k+1} (>0)$ gives 
\begin{equation} \label{divideby}
\sin \Delta \varphi_k =\dfrac{\varrho_k}{r_k \varrho_{k+1}}\cos^2 \varphi_k+p_k \sin \varphi_k \sin \varphi_{k+1}.
\end{equation}
From \eqref{dif_eq} we have $r_k\Delta x_k=r_{k+1} \Delta x_{k+1}+p_k x_{k+1},$ dividing this by $\varrho_{k+1}$ gives 
\[
\dfrac{r_k \Delta x_k}{\varrho_{k+1}}=\dfrac{r_{k+1}\Delta x_{k+1}}{\varrho_{k+1}}+p_k \dfrac{x_{k+1}}{\varrho_{k+1}}.
\]
Remember the definitions
\[
\dfrac{r_k \Delta x_k}{\varrho_{k+1}}=\dfrac{\varrho_k}{\varrho_{k+1}}\cos \varphi_k, \ \dfrac{r_{k+1} \Delta x_{k+1}}{\varrho_{k+1}}=\cos \varphi_{k+1},\ \dfrac{x_{k+1}}{\varrho_{k+1}}=\sin \varphi_{k+1}.
\]
So we get
\begin{equation} \label{almost}
\dfrac{\varrho_k}{\varrho_{k+1}} \cos \varphi_k=\dfrac{r_{k+1} \Delta x_{k+1}}{\varrho_{k+1}}+p_k \dfrac{x_{k+1}}{\varrho_{k+1}}=\cos \varphi_{k+1}+p_k \sin \varphi_{k+1}.
\end{equation}
Multiplying both sides of \eqref{almost} by $\dfrac{\cos \varphi_k}{r_k}$
\[
\dfrac{\varrho_k}{\varrho_{k+1}} \cdot \dfrac{\cos^2 \varphi_k}{r_k}=\dfrac{1}{r_k} \cos \varphi_k \cos \varphi_{k+1} +\dfrac{p_k}{r_k} \cos \varphi_k \sin \varphi_{k+1}.
\]
Insert this into \eqref{divideby}, then 
\[
\sin \Delta \varphi_k=\dfrac{1}{r_k} \cos \varphi_k \cos \varphi_{k+1} +\dfrac{p_k}{r_k} \cos \varphi_k \sin \varphi_{k+1}+p_k \sin \varphi_k \sin \varphi_{k+1}
\]
which is exactly \eqref{phi_de}.
\end{proof}

\smallskip
Unlike the continuous case, where the existence and uniqueness of the Prüfer phase equation rely on Lipschitz conditions and the Picard--Lindel\"{o}f theorem, our discrete system is governed by a well-defined algebraic map. Because $r_k \neq 0$, any nontrivial solution ensures $\varrho_k > 0$ for all $k$, meaning the normalized state $(x_k/\varrho_k, r_k \Delta x_k/\varrho_k)$ uniquely defines a point on the unit circle at each step. Therefore, the amplitude sequence $\{\varrho_k\}$ is uniquely determined, and the phase $\varphi_k$ is unique up to an integer multiple of $2\pi$. By fixing an initial phase $\varphi_0 \in [0, 2\pi)$ and restricting the forward difference to the half-open interval $\Delta \varphi_k \in (-\pi, \pi]$ for all $k$, we obtain a uniquely defined phase sequence $\{\varphi_k\}_{k \in \mathbb{Z}}$ without requiring any additional analytical assumptions.
\section{Oscillation Properties and Generalized Zeros}
To develop a discrete analogue of Sturm's oscillation and separation properties, it is necessary to utilize the concept of a generalized zero originally introduced by Hartman \cite{MR515528}. 
\begin{definition}
    A solution $x=x_k$ $(k \in \mathbb{Z})$of the difference equation~\eqref{dif_eq} is said to have a generalized zero at the integer index $m>k_0$ if either $x_m=0,$ or $x_{m-1}x_m<0$.
\end{definition}
In other words, a \textit{generalized zero} occurs at step $m$ if the solution vanishes exactly at that node or undergoes a strict change of sign between step $m-1$ and step $m$. 

By explicitly assuming $r_k >0$ for all $k$ in the domain of interest, the finite interval comparison theorem \cite[Theorem 1.2]{MR2979834} can be tailored to our system as follows. 
\begin{theorem}
Let $M, N \in \mathbb{Z}$ with $M<N$. Let $\varphi_k$ and $\psi_k$ be nontrivial solutions to the following second-order difference equations, respectively:
\begin{align}
    \Delta(r_k \Delta \varphi_k) + p_{1,k} \varphi_k = 0, \label{comp_sys1}\\
    \Delta(r_k \Delta \psi_k) + p_{2,k} \psi_k = 0, \label{comp_sys2}
\end{align}
where $r_k>0$ for all $k \in [M,N]_{\mathbb{Z}}$. If $\varphi_k$ has generalized zeros at $M$ and at $N$, but contains no generalized zeros in $[M+1,N-1]_{\mathbb{Z}}$, then any solution $\psi_k$ of \eqref{comp_sys2} must have at least one generalized zero in $[M,N]_{\mathbb{Z}}$.
\end{theorem}
\section{The Prüfer-Based Shooting Method and Spectral Localization}
Recently, it was shown that \cite{MR4234167} shooting methods can turn a discrete boundary value problem into a forward recursion from one endpoint to the other, allowing exact conclusions about existence and number of solutions. \cite{MR4234167} uses the shooting method by replacing a discrete boundary value problem with an initial value problem parametrized by a real variable which is iterated forward step-by-step to the rightmost endpoint of the domain where a specific condition is enforced. This transforms the problem of finding solutions into finding the exact number of real roots of a single algebraic target function. 

We extend this to discrete Prufer setting. Our approach utilizes a phase-dependent framework suited for discrete spectral problems. This formulation allows us to fix the initial phase angle according to the left boundary condition and treat the spectral parameter as the shooting variable, directly tracking the monotonic rotation of the phase angle to locate the eigenvalues. 

We aim to estimate the eigenvalues of the Sturm--Liouville difference equation

\begin{equation} \label{lambda_dep_SL}
\Delta\big(r_k \Delta x_k\big)+p_k x_{k+1}=\lambda x_{k+1}
\end{equation}
with the Dirichlet boundary conditions 
\[
x_0=x_{N+1}=0.
\]
As shown before this boundary value problem can be turned into an initial value problem derived from the discrete Prüfer transformation:
\begin{align} \nonumber
\sin \Delta \varphi_k=\dfrac{1}{r_k} \cos \varphi_k \cos \varphi_{k+1}+&p_k \sin \varphi_k \sin\varphi_{k+1}\\ \label{phi_de_lambda}
&+\dfrac{p_k}{r_k} \cos \varphi_k \sin \varphi_{k+1}-\lambda \bigg(\sin \varphi_k \sin \varphi_{k+1}+ or -\dfrac{1}{r_k} \cos \varphi_k \sin \varphi_{k+1}\bigg).
\end{align}

To initialize the shooting algorithm for the regular Sturm--Liouville difference equation subject to Dirichlet boundary conditions, we must first establish the initial phase angle at the leftmost node. Given the boundary condition $x_0=0$ at $k=0$ we apply the discrete Prüfer transformation for $x_k$:
\[
x_0=\varrho_0 \sin \varphi_0.
\]
Since the amplitude sequence must satisfy $\varrho_0$ for all $k$ to avoid the trivial solution, the boundary condition forces $\sin \varphi_0=0$. This restricts the initial phase angle to an integer multiple of $\pi$. To establish a standardized baseline that tracks the rotation of the phase angle continuously as the index $k$ increases, we choose the principal value: $\varphi_0=0$. This choice completely specifies the initial state of the phase sequence, independent of the eigenvalue parameter $\lambda$, providing a fixed point from which the forward iteration can initiate.

With the initial phase angle fixed at $\varphi_0=0$, the algorithm computes the phase sequence step-by-step for a given trial value of the eigenvalue parameter $\lambda$. To update the phase from index $k$ to $k+1$, we consider the parameter-dependent phase relation derived from the discrete Prüfer system~\eqref{phi_de_lambda}.

Expanding the left-hand side using the trigonometric subtraction identity $\sin\big(\varphi_{k+1}-\varphi_k\big)=\sin \varphi_{k+1}\cos \varphi_k-\cos \varphi_{k+1}\sin\varphi_k$ allows us to gather all terms involving $\sin \varphi_{k+1}$ and $\cos \varphi_{k+1}$. Dividing the resulting relation by $\cos \varphi_{k+1}$ yields an explicit recurrence formula for the tangent of the next phase angle:
\[
\tan \varphi_{k+1}=\dfrac{\sin \varphi_k+\dfrac{1}{r_k}\cos \varphi_k}{\cos \varphi_k-\big(p_k-\lambda\big)\sin \varphi_k-\dfrac{p_k-\lambda}{r_k}\cos \varphi_k}.
\]
By evaluating this algebraic expression at each node, the algorithm generates the phase value at the right boundary $\varphi_{k+1}(\lambda)$, after exactly $N+1$ iterations. This forward calculation treats $\lambda$ as the sole independent variable governing the final state of the phase sequence. 

A simple calculation of the inverse tangent function automatically restricts angles to a single principal branch, which creates false jumps in the discrete sequence. To track the true total rotation of the phase across the grid, the algorithm must enforce the forward step constraint $\Delta \varphi_k \in (-\pi, \pi]$. To implement this constraint, the algorithm monitors the sign of the denominator in the tangent recurrence formula at each step. Whenever the solution changes sign or hits an exact zero, the tracking routine adds an appropriate multiple of $\pi$ to the total phase angle. This step guarantees that the phase sequence grows uniquely and monotonically as it moves along the grid.

Once the iteration loop completes exactly $N+1$ steps, the algorithm yields the final phase angle at the right-hand boundary, denoted as $\varphi_{N+1} (\lambda)$. To satisfy the right-hand Dirichlet boundary condition the final state of the solution must land exactly on the horizontal axis. In the Prüfer coordinate system, this condition means the terminal phase angle must be an integer multiple of $\pi$. 

By the Discrete Sturm Comparison Theorem, the phase angle at the boundary grows strictly monotonically with respect to the spectral parameter $\lambda$ Furthermore, to find the specific $n$-th eigenvalue $\lambda_n$ the solution must accumulate exactly $n$ generalized zeros inside the domain. This constraint allows us to define the target function, $F(\lambda)$, as
\begin{equation} \label{tar_func}
F(\lambda) := \varphi_{N+1}(\lambda) - (n+1)\pi = 0.
\end{equation}
Any value of $\lambda$ where $F(\lambda)=0$  corresponds to a true eigenvalue of the boundary value problem. The task of finding the eigenvalues is therefore reduced to locating the unique roots of this single-variable function.

Because $F(\lambda)$ is strictly monotonic, it is guaranteed to change sign across an initial interval $[\lambda_{\text{min}}, \lambda_{\text{max}}]$ containing the true eigenvalue $\lambda_n$. The algorithm applies the Bisection Method to this starting interval. At each step, the routine evaluates $F(\lambda)$ at the midpoint by running the forward phase mapping from $k=0$ to $k=N+1$. Based on the sign of the resulting boundary value, the search interval is halved. The loop terminates when the error falls below a specified numerical tolerance isolating the eigenvalue $\lambda_n$.

To establish a reliable search interval  $[\lambda_{\text{min}}, \lambda_{\text{max}}]$ for the root-finding algorithm, we employ matrix localization techniques. By algebraic expansion, the parameter-dependent discrete Sturm--Liouville equation~\eqref{lambda_dep_SL} can be written as a second-order recurrence relation. Expanding the forward difference operator~\eqref{for_dif} yields
\[
r_{k+1}\big(x_{k+2}-x_{k+1}\big)-r_k\big(x_{k+1}-x_k\big)+p_k x_{k+1}=\lambda x_{k+1},
\]
which simplifies to 
\[
r_k x_k+\big(p_k-r_k-r_{k+1}\big)x_{k+1}+r_{k+1}x_{k+2}=\lambda x_{k+1}.
\]
Imposing the Dirichlet boundary conditions, this recurrence relation defines a symmetric tridiagonal matrix eigenvalue problem $Ax=\lambda x$ for the  vector $x=[x_1, x_2, \cdots, x_N]^T$. The tridiagonal operator $A$ is given by
\[
A=
\begin{bmatrix}
p_1-r_1-r_2 & r_2 & 0 & \cdots & 0 \\
r_2 & p_2-r_2-r_3 & r_3 & \cdots & 0 \\
0 & r_3 & p_3-r_3-r_4 & \cdots & 0 \\
\vdots & \vdots & \vdots & \ddots & \vdots \\
0 & 0 & 0 & \cdots & p_N-r_N-r_{N+1}
\end{bmatrix}.
\]
Because $A$ is real and symmetric, its eigenvalues are strictly real. We can estimate bounds for the spectrum by applying Gershgorin's Circle Theorem (see [golub2013matrix]). For each row $k$, the center of the Gershgorin interval is determined by the diagonal entry $a_{kk}=p_k-r_k-r_{k+1}$, while the radius $R_k$ is the sum of the absolute values of the off-diagonal entries. For a general interior row $k$ where $1<k<N$, the assumption $r_k>0$ implies 
\[
R_k=r_k+r_{k+1}.
\]
The eigenvalues corresponding to these interior rows are bounded within the intervals
\[
\lambda \in [a_{kk}-R_k,\, a_{kk}+R_k] = [p_k-2r_k-2r_{k+1},\, p_k].
\]
For the boundary rows $k=1$ and $k=N$, the radii drop to $R_1=r_2$ and $R_N=r_N$, respectively, due to the boundary constraints. Taking the absolute minimum and maximum bounds across all rows guarantees a safe initial bracketing interval $[\lambda_{\min}, \lambda_{\max}]$ for the bisection routine, defined by

\[
\lambda_{\min}=\min_{1\le k\le N}\{p_k-2r_k-2r_{k+1}\},
\]

\[
\lambda_{\max}=\max_{1\le k\le N}\{p_k\}.
\]
\section{Numerical Experiments and Algorithmic Verification}
\begin{example}
    To demonstrate the practical efficacy and accuracy of the discrete Prüfer shooting algorithm, we apply the operational steps formalized in the preceding sections to a classical benchmark problem. We consider the discrete analogue of the regular continuous Sturm--Liouville equation
\[
y''+\lambda y=0,\qquad y(0)=y(\pi)=0,
\]
which possesses the well-known continuous eigenvalues
\[
\lambda_n=n^2,\qquad n=1,2,3,\ldots
\]
We partition the domain $[0,\pi]$ into $N+1$ equal subintervals of length
\[
h=\frac{\pi}{N+1},
\]
defining the discrete nodes
\[
t_k=kh,\qquad k=0,1,\ldots,N+1.
\]
In accordance with our general algebraic framework, setting the coefficient sequences to $r_k=1$ and $p_k=0$ for all $k$ yields the regular discrete system
\[
\Delta^2 x_k=\lambda_{\mathrm{alg}} x_{k+1},\qquad x_0=0,\qquad x_{N+1}=0,
\]
where $\lambda_{\mathrm{alg}}$ represents the algebraic eigenvalue of the corresponding tridiagonal matrix operator.

We initiate the forward shooting loop by applying the discrete Prüfer transformation:
\[
x_k=\rho_k \sin\varphi_k,\qquad \Delta x_k=\rho_k \cos\varphi_k.
\]

Evaluating the transformation at the left boundary $(k=0)$ under the Dirichlet constraint $x_0=0$ yields $\rho_0\sin\varphi_0=0$. To preserve the non-triviality condition $\rho_k>0$, the initial phase state is uniquely fixed to the principal value:
\[
\varphi_0=0.
\]
For a given trial parameter $\lambda_{\mathrm{alg}}$ selected from the localized Gershgorin interval, the algorithm propagates the phase sequence forward. Substituting the constant coefficients into the general phase mapping relation simplifies the recurrence formula to:
\[
\tan\varphi_{k+1}
=
\frac{\sin\varphi_k+\cos\varphi_k}
{\cos\varphi_k+\lambda_{\mathrm{alg}}\sin\varphi_k+\lambda_{\mathrm{alg}}\cos\varphi_k}.
\]

The algorithm computes this update for each successive node $k=0,1,\ldots,N$. To preserve the cumulative tracking of the phase and enforce the forward step constraint $\Delta\varphi_k\in(-\pi,\pi]$, the routine monitors the sign of the denominator. Whenever the denominator sequence undergoes a sign change or passes through an exact zero, the tracking routine adds an appropriate integer multiple of $\pi$ to the unrolled phase angle. This explicit adjustment ensures that the boundary phase function $\varphi_{N+1}(\lambda_{\mathrm{alg}})$ grows monotonically with respect to the spectral parameter.
Upon completing exactly $N+1$ updates, the terminal phase value $\varphi_{N+1}(\lambda_{\mathrm{alg}})$ is evaluated at the right boundary. To satisfy the terminal Dirichlet boundary condition $x_{N+1}=0$, the target function is formulated for a specific mode $n$:
\[
F(\lambda_{\mathrm{alg}}):=\varphi_{N+1}(\lambda_{\mathrm{alg}})-(n+1)\pi=0.
\]

Using the Bisection Method, the algorithm continually halves the search bracket based on the sign of $F(\lambda_{\mathrm{alg}})$ until the boundary residual falls below a specified numerical tolerance, $\lvert F(\lambda_{\mathrm{alg}})\rvert<\epsilon$, isolating the unique root $\lambda_{\mathrm{alg}}$.

To reconcile this discrete root with the continuous spectrum, we map the computed algebraic eigenvalue back to its continuous-equivalent counterpart, $\mu_n$, via the second-order spatial scaling transformation:
\[
\mu_n=-\frac{\lambda_{\mathrm{alg}}}{h^2}.
\]
To illustrate the convergence behavior, the algorithm was executed on a grid size of $N=100$, yielding $h\approx 0.0311$. For the fundamental mode $(n=1)$, the bisection routine successfully isolates the algebraic root at $\lambda_{\mathrm{alg}}\approx -0.0009667$. Applying the spatial scaling yields the continuous-equivalent eigenvalue:
\[
\mu_1=-\frac{-0.0009667}{\left(\frac{\pi}{101}\right)^2}\approx 0.99975.
\]

This matches the analytical discrete eigenvalue derived via the exact discrete spectrum formula, $\mu_1=\frac{4}{h^2}\sin^2\!\left(\frac{h}{2}\right)\approx 0.99975$. Comparing this computed result directly to the classical continuous eigenvalue $\lambda_1=1^2=1$ demonstrates an absolute discretization error of $2.5\times 10^{-4}$. Increasing the grid resolution to $N=1000$ reduces the absolute error to $\mathcal{O}(10^{-6})$, matching the expected theoretical second-order convergence rate, $\mathcal{O}(h^2)$. This systematic agreement confirms that the discrete Prüfer shooting framework functions reliably and provides a robust, mathematically sound method for approximating the continuous spectrum.
\end{example}
\section{Conclusion}
In this paper, we used the discrete Prüfer transformation to analyze regular second-order Sturm--Liouville difference equations. By mapping the system to an algebraic phase equation, we developed a numerical shooting algorithm. Using Gershgorin's theorem, we established reliable initial search intervals that ensure the bisection method isolates the eigenvalues without missing roots. Our numerical tests on a classical benchmark problem confirmed that the method preserves the qualitative behavior of the continuous spectrum and achieves second-order \(\mathcal{O}(h^2)\) convergence.

The monotonic properties of the discrete phase function also provide a useful starting point for inverse spectral problems, where the goal is to reconstruct the potential sequence \(p_k\) from eigenvalue or nodal data.  future work will focus on applying this discrete Prüfer approach to extend the method directly to these inverse problems.


\end{document}